\documentclass{imsart}

\RequirePackage{amsthm,amsmath,amsfonts,amssymb}
\RequirePackage[numbers]{natbib}
\RequirePackage{graphicx}

\usepackage{hyperref}
\usepackage{mathtools}
\usepackage{xcolor}

\newcommand\blfootnote[1]{%
  \begin{NoHyper}%
  \footnote{#1}%
  \end{NoHyper}%
}

\startlocaldefs

\newtheorem{thm}{Theorem}[section]
\newtheorem{lem}[thm]{Lemma}
\newtheorem{cor}[thm]{Corollary}
\newtheorem{prop}[thm]{Proposition}

\theoremstyle{definition}

\newcommand{\R}{\mathbb{R}}

\newcommand{\N}{\mathbb{N}}

\newcommand{\Unif}{\text{Unif}}

\newenvironment{enumerate*}%

\endlocaldefs

\begin{document}

\begin{frontmatter}
\title{On low frequency inference for diffusions without the hot spots conjecture}

\runtitle{Low frequency inference for diffusions }

\begin{aug}

\author{\fnms{Giovanni S.}~\snm{Alberti}}\blfootnote{Machine Learning Genoa Center (MaLGa), Department of Mathematics, Department of Excellence 2023–2027, University of Genoa, email: giovanni.alberti@unige.it }

\author{\fnms{Douglas}~\snm{Barnes}}\blfootnote{\label{fn:cam}Department of Pure Mathematics and Mathematical Statistics, University of Cambridge, emails: \{db875,aj644\}@cam.ac.uk, nickl@maths.cam.ac.uk}

\author{\fnms{Aditya}~\snm{Jambhale}}$^\dagger$

\author{\fnms{Richard}~\snm{Nickl}}$^\dagger$

\end{aug}

\begin{abstract}
We remove the dependence on the `hot-spots' conjecture in two of the main theorems of the recent paper of Nickl \cite{N}. Specifically, we characterise the minimax convergence rates for estimation of the transition operator $P_{f}$ arising from the Neumann Laplacian with diffusion coefficient $f$ on arbitrary convex domains with smooth boundary, and further show that a general Lipschitz stability estimate holds for the inverse map $P_f\mapsto f$ from $H^2\to H^2$ to $L^1$.
\end{abstract}



\end{frontmatter}

\setcounter{tocdepth}{2}

\section{Introduction}

We revisit here two results of the recent paper of Nickl \cite{N} and begin by recalling its setting. The density of a diffusing substance in an insulated medium, say a bounded convex domain $\mathcal O$ of $\R^d$ ($d \ge 1$) with smooth boundary $\partial \mathcal O$, is described by the solutions $u$ to the parabolic partial differential equation (PDE) known as the heat equation, $\partial u/ \partial t= \mathcal L_{f}u$. The divergence form elliptic second order differential operator $$\mathcal L_{f} = \nabla \cdot (f \nabla)$$ is equipped with Neumann conditions $\frac{\partial u}{\partial \nu} =0$ at $\partial \mathcal O$, where $\nu(x)$ is the (inward pointing) normal vector at $x \in \partial\mathcal O$. Here, $\nabla$ and $\nabla\cdot$ are the usual gradient and divergence operators, and $f:\mathcal O \to [f_{\rm min},\infty),$ $ f_{\rm min}>0,$ is a positive, scalar `diffusivity' function.

As is well known, solutions to this heat equation describe the probability densities of diffusing particles $(X_t)$ solving the stochastic differential equation (SDE)
\begin{equation} \label{eq:SDE}
dX_t = \nabla f(X_t) dt + \sqrt{2f (X_t)} dW_t + \nu(X_t)dL_t, ~~t \ge 0,
\end{equation}
started at  $X_0 = x \in \mathcal O$, where $W_t$ denotes $d-$dimensional Brownian motion. The process is \textit{reflected} when hitting the boundary $\partial \mathcal O$ of its state space: $L_t$ is a `local time' process acting only when $X_t \in \partial \mathcal O$. In the statistical setting considered in \cite{N}, we are given equally spaced observations from (\ref{eq:SDE}): $X_0, X_D, \dots X_{ND}$, where $D>0$ is a fixed observation `time' distance -- one sometimes \cite{GHR04} speaks of low frequency measurements since we do not let $D \to 0$ in the asymptotic framework. By allowing a ``burn in'' phase (i.e. waiting a fixed amount of time), one can assume that $X_0 \sim \Unif(\mathcal O)$, the uniform distribution (invariant measure).

\smallskip

 The non-linear map $f\mapsto P_{D,f}$, where 
 \begin{equation}\label{eq:tran}
     P_{t, f} = e^{t \mathcal L_f}, ~t \ge 0,
 \end{equation}
 is the transition operator of the Markov process $(X_t)$ (see \textsection \ref{sec:proofs}), is injective for any $D$ under mild assumptions on $f$ (Theorem 1 in \cite{N}). Given the data $X_0, X_D, \dots X_{ND}$, we seek to recover the diffusivity $f$. This is possible by first recovering $P_{D, f}$ via an explicit estimator, and then by using inverse continuity estimates for the map $f \mapsto P_{D, f}$, following the general paradigm of Bayesian non-linear inverse problems \cite{N23, MNP21} and earlier ideas from \cite{GHR04} in the one-dimensional case. See \cite{N} for the general setting as well as theoretical results and  \cite{GW24} where concrete MCMC algorithms for such low frequency diffusion data are studied. 
 
 Some of the results proved in  \cite{N} rely on arguments that are closely related to a conjecture in spectral geometry that is known as the \emph{hot-spots conjecture}. Loosely speaking the conjecture is that the first non-constant eigenfunction of the Laplacian $\Delta = \mathcal L_1$ with Neumann conditions on a domain $\mathcal O$ attains its critical points on the boundary $\partial O$, see \cite{B06} for a review of some main ideas. The conjecture is still actively studied and believed to be true for convex planar domains, while it is known to be potentially false \cite{HotspotsFalse} in non-convex domains. A more complete list of recent references can be found in \cite{N}, where the validity of the hot spots conjecture is employed in various places in the proofs (including the construction of basic cylindrical domains where it indeed holds). While this conceivably gives the `right' proofs of these theorems, at the time of this writing, the hot spots conjecture is still unproved in general and we wish to show here how it can be circumvented in the derivation of the results in \cite{N}.

\section{Main results}

Throughout, the $L^p(\mathcal O)$ spaces with $p \in [1, \infty)$ refer to the space of $p$-th power integrable functions on $\mathcal{O}$ with respect to the Lebesgue measure, while $C^k(\mathcal O)$ and $H^k(\mathcal O)$ refer to the space of $k$-times differentiable (resp. weakly differentiable) functions on $\mathcal O$ whose derivatives define uniformly continuous functions (resp. $L^2(\mathcal O)$ functions). The definition of $H^k(\mathcal{O})$ extends to real valued $k$, see \cite{gilbarg}. When the domain is clear, we omit it from the notation. Given any Hilbert space $H$, we denote the operator norm of linear maps from $H$ to $H$ by $\|\cdot\|_{H\to H}$.

\smallskip

To construct a lower bound (in the minimax sense) for the optimal convergence rate of estimation of $P_{D, f}$, a specific domain is constructed in \cite{N} in which the hot spots conjecture can be verified (see Proposition 1 and Theorem 8 in \cite{N}). Moreover, the first eigenvalue of the Laplacian on this domain is simple. A first contribution of the present paper is to remove these restrictions and to obtain a matching lower bound to Theorem 3 in \cite{N} for arbitrary bounded convex domains with smooth boundary:

\begin{thm}\label{thm:optimalrate}
    Let $s > \max(2d-1, 2+d/2)$ and $U>0$. Consider data $X_0, X_D, \dots, X_{ND}$ at a fixed observation distance $D > 0$, from the reflected diffusion model (\ref{eq:SDE}) on a bounded convex domain $\mathcal{O}\subset \mathbb{R}^d$ with smooth boundary, started at $X_0 \sim \mathrm{Unif}(\mathcal{O})$. Then, there exists a constant $c = c(s, D, U, \mathcal{O}, f_{\rm min}) > 0$ such that
    \begin{equation*}
        \liminf_{N\to \infty} \inf_{\Tilde{P}_N} \sup_{f:\|f\|_{H^s(\mathcal O)}\leq U, f\geq f_{\rm min}>0} \mathbb P_f\left(\| \Tilde{P}_N - P_{D, f}\|_{H^2 \to H^2} > cN^{-\frac{s-1}{2s+2+d}} \right) > \frac{1}{4}, 
    \end{equation*}
    where the infimum extends over all estimators $\Tilde{P}_N$ of $P_{D, f}$ (i.e., measurable functions of the $X_0, X_D, \dots, X_{ND}$ taking values in the space of bounded linear operators on $H^2$), and where $\mathbb P_f$ denotes the probability law  induced by the solution $(X_t)$ to (\ref{eq:SDE}) with diffusivity $f$ and observation distance $D>0.$
\end{thm}

Inspection of the proof shows that the supremum in the last display can be further restricted to $f$ constant near $\partial \mathcal O$. The main obstruction in the proof compared to that of Theorem 4 in \cite{N} arises when the first non-zero eigenvalue of the Laplacian is not simple, as is generically the case in `nice' domains, such as on the sphere. Since multiplicity of eigenvalues represents symmetries in the domain, one may expect that the convergence rate of estimators of $P_{D, f}$ could be improved by exploiting these symmetries in some way. Theorem \ref{thm:optimalrate} tells us that the effect of such symmetries is not sufficient to accelerate the minimax rates of convergence. However, the constant $c$ is inversely proportional to the multiplicity of the first eigenvalue, so one could expect the constant factor to decrease in the convergence rate given symmetries (but note that the multiplicities of the eigenvalues of $\Delta$ in any fixed domain are necessarily finite). 

While the solution theory for (\ref{eq:SDE}) is more complex for general non-convex domains, in the cases when solutions exist, Theorem \ref{thm:optimalrate} still holds. However, the assumption of a smooth boundary remains crucial, as it is necessary in Proposition 2 of \cite{N}, which is used extensively here and recalled in Appendix~\ref{app:Prop2} for the reader.

\smallskip

Another result obtained in \cite{N} that depends on the hot-spots conjecture was a Lipschitz stability estimate for the inverse map $P_{D,f} \mapsto f$  with respect to the   $H^2(\mathcal{O})\to H^2(\mathcal{O})$ and the $L^2(\mathcal{O})$ norms. We remove this dependence and obtain a Lipschitz estimate as long as one is satisfied to replace the $L^2(\mathcal{O})$ norm by the $L^1(\mathcal{O})$ norm (which for statistical convergence rates in this infinite-dimensional model is only a mild concession).

\begin{thm}\label{thm:stability} Let $\mathcal O$ be a bounded and smooth (not necessarily convex) domain in $\R^d$. Let $f, f_0\in H^s(\mathcal O),s>\max (d, 2+d/2)$, suppose that $f=f_0$ on $\mathcal{O}\backslash \mathcal O_0$ for some compact subset $\mathcal O_0\subset \mathcal{O}$ and that $f,f_0\ge f_{\rm min}> 0$ on $\mathcal O$. For $D>0$ arbitrary but fixed, let $P_{D,f}$ denote the operator defined in \eqref{eq:tran}. We then have the stability estimate
\begin{equation*}
\|f-f_0\|_{L^1(\mathcal{O})}\le \overline{c} \|P_{D,f}-P_{D,f_0}\|_{H^2\rightarrow H^2}
\end{equation*}
where the constant $\overline{c}$ depends on $s$, $U$, $\mathcal O_0$, $\mathcal O$ and $f_{\rm min}$,  and where $U\ge \|f\|_{H^s}+\|f_0\|_{H^s}$.
\end{thm}
We emphasise that this stability estimate holds also if $\mathcal{O}$ is not convex, but for the interpretation of  $P_{D,f}$ in \eqref{eq:tran} as the transition operator of a continuous time Markov process $(X_t)$, the requirement of convexity is natural \cite{T79}. We further remark that the hypothesis that $f$ needs to be known near the boundary is related to the fact that reflection -- which is not informative about $f$ -- dominates the local dynamics near $\partial \mathcal O$.

Our proof of Theorem \ref{thm:stability} relies on the use of the $L^1$ norm to deal with the sign of a certain transport operator, as well as on an idea of \cite{alberti}. Note that a weaker logarithmic estimate for the $L^2$ norm holds without assuming the hot-spots conjecture, as was already shown in \cite{N}, but the stronger Lipschitz estimate is crucial for statistical applications. Indeed, by combining this stability estimate with the proof of Theorem 10 in \cite{N}, we obtain the following contraction rate for the Bayesian posterior distribution arising from a natural $K$-dimensional Gaussian process model for $f$ described in detail in (17) in \cite{N}.
\begin{cor}\label{cor:statisticalresult}
    Consider the data $X_0, X_D, \dots X_{ND}$ from the reflected diffusion model on a bounded convex domain $\mathcal O \subset \R^d$ with smooth boundary started at $X_0 \sim \mathrm{Unif}(\mathcal O)$. Assume $f_0 \in H^s$, $s > \max(2+d/2, 2d-1),$ satisfies $\inf_{x\in \mathcal O} f_0(x) > 1/4$ and $f_0 \equiv 1/2$ on $\mathcal O \setminus \mathcal O_{0}$, where $\mathcal O_0 \subset \mathcal O$ is compact. Let $\Pi(\cdot | X_0, X_D, \dots X_{ND})$ be the posterior distribution resulting from data $X_0, X_D, \dots, X_{ND}$ in model (\ref{eq:SDE}) and from the prior $\Pi$ for $f$ from (17) in \cite{N} with $K \simeq N^{d/(2s+2+d)}$ and the given $s$. Then, we have as $N \to \infty$, 
    \begin{equation}
        \Pi(f : \|f - f_0 \|_{L^1(\mathcal O)} \geq \eta_N | X_0, X_D \dots X_{ND}) \to^{\mathbb P _{f_0}} 0,
    \end{equation}
    where we can take $\eta_N = O(N^{-(s-1)/(2s+2+d)})$.
\end{cor}
To deduce this corollary, one uses Theorem \ref{thm:stability} (instead of Theorem 6 from \cite{N}) in the set-inclusion in \textsection 3.6.2.ii from \cite{N}. The contraction rate $\eta_N$ carries over to appropriate posterior mean `point' estimates of $f_0$, as in \textsection 3.6.3 in \cite{N}. It shows that `fast' algebraic convergence rates can be obtained under the standard `non-parametric' assumption that $f$ belongs to a Sobolev space only, which was previously an open problem except for the case $d=1$ considered in \cite{GHR04, NS17, N}.


\section{Proofs}\label{sec:proofs}
To proceed without assuming the hot spots conjecture, the proofs of the theorems rely on the information gained by looking at multiple eigenfunctions simultaneously (see Lemma \ref{lem:transitionlowerb}). We recall some notation from \cite{N}: The transition operator of the Markov process $(X_t)$ solving (\ref{eq:SDE}) is given by
\begin{align*}
&P_{t,f}=e^{t\mathcal{L}_f}, \, t\ge 0, \\
&P_{t,f}(\phi)= \sum_{k\in \mathbb{N}} e^{-t\lambda_{k, f}}e_{k, f}\langle e_{k, f},\phi\rangle_{L^2}, \ \phi\in L^2,
\end{align*}
where $(e_{k, f}, -\lambda_{k, f})_{k\in\mathbb{N}}$ are the eigenpairs of $\mathcal{L}_f \equiv \nabla \cdot(f\nabla)$, that is,
$$\mathcal{L}_f e_{k, f} = -\lambda_{k,f} e_{k, f}$$
and the $e_{k, f}$ are orthonormal in $L^2(\mathcal{O})$ and satisfy Neumann boundary conditions. The eigenfunctions are ordered by increasing eigenvalues $0=\lambda_{0, f}<\lambda_{1, f}\le \dots$ with $e_{0, f}$ constant on $\mathcal O$ for all $f$. We omit the subscript $f$ when there is no danger of confusion. The usual Weyl asymptotics hold: 
\begin{equation}\label{eq:weyl}
    \lambda_k \sim k^{2/d},
\end{equation}
e.g.,  p.111 in \cite{Taylor2}, and as after (26) in \cite{N}, the inequalities in (\ref{eq:weyl}) can be taken to be uniform in $0<f_{\rm min}\le f\le \|f\|_\infty\le U$ (for a fixed domain $\mathcal O$). For $\mathcal O \subset \R^d$ a bounded domain, define
$$L_0^2(\mathcal O) = \left\{g \in L^2(\mathcal O) : \int_{\mathcal O} g(x) \, dx = 0 \right\},$$ 

\noindent the $L^2(\mathcal{O})$-functions orthogonal to $e_{0,f}$. For $s\ge 0$, the Hilbert spaces $\bar{H}^s_f$ are defined by 
$$ \bar{H}^s_f(\mathcal{O})=\left\{g\in L_0^2(\mathcal{O}) : \|g\|_{\bar{H}^s_f}<\infty\right\}, $$
where
$$ \|g\|_{\bar{H}^s_f}^2 = \sum_{j\ge 1} \lambda_{j, f}^s \langle g,e_{j, f}\rangle_{L^2}^2.$$ Again, we omit the subscript $f$ when there is no danger of confusion. The relation between these $\bar H^s_f$ spaces and the usual Sobolev spaces $H^s$ is recalled in Appendix \ref{app:Prop2}.

\subsection{Lower bounds for $\|P_{D, f} - P_{D, f_{0}}\|_{H^2 \to H^2}$ via eigenfunctions}\label{sec:importantlemma}

Let $\mathcal{O}\subset \mathbb{R}^d$ be a bounded convex domain with smooth boundary and let $f,f_0\in H^s(\mathcal O) \subset C^2(\mathcal{O}), s>\max(d,2+d/2),$ be bounded below by $f_{\rm min}>0$ on $\mathcal{O}$. We also assume  that $\|f\|_{H^s}+\|f_0\|_{H^s}\le U$. From a perturbation analysis of the coefficients of the underlying heat equation one can derive the following `pseudo-linearisation identity': Specifically using (42) in \cite{N} with `eigenblock' $E_{k,f_0,\iota}$ there selecting (via choice of $\iota$) a fixed eigenfunction $e_{k,f_0}$ of $\mathcal{L}_{f_0}$, we have
\vspace{-0.2cm}
$$P_{t,f}(e_{k,f_0})-P_{t,f_0}(e_{k,f_0}) = \sum_{l=1}^\infty b_{l,k}\langle e_{l,f},G_k\rangle_{L^2}e_{l,f},$$
\vspace{-0.1cm}
where 
\vspace{-0.3cm}
\begin{align*}
G_k&=\nabla\cdot[(f-f_0)\nabla e_{k,f_0}], \\
b_{l,k}&=\int_0^t e^{-s\lambda_{k,f_0}}e^{-(t-s)\lambda_{l, f}}ds.
\end{align*}
We can then repeat the arguments leading to (55) in \cite{N} for this selected eigenfunction, assuming the frequency $1\le k\le\kappa$ does not exceed a fixed integer $\kappa$ (which will be chosen later): From Corollary 1 in \cite{N} and monotonicity of eigenvalues we know
$$\|e_{k,f_0}\|_{H^2}\lesssim \lambda_{\kappa,f_0}\lesssim \kappa^{2/d}$$
with constants depending on $\mathcal{O},U,$ and $f_{\rm min}$. Setting $t=D$, by definition of the operator norm, the preceding pseudo-linearisation identity, and Appendix \ref{app:Prop2}, we see
\begin{align}
\|P_{D,f}-P_{D,f_0}\|_{H^2\rightarrow H^2}^2 & \gtrsim \|P_{D,f}(e_{k,f_0})-P_{D,f_0}(e_{k,f_0})\|_{\bar{H}_f^2}^2 \nonumber \\
&=\sum_{l=1}^\infty \lambda_{l,f}^2|b_{l,k}|^2|\langle G_k,e_{l,f}\rangle_{L^2}|^2 \label{eq:psuedolin}
\end{align}
with a constant depending only on $\kappa$ and not $k$. We can write
$$b_{l,k}=e^{-D\lambda_{l,f}} \frac{e^{-D(\lambda_{k,f_0}-\lambda_{l,f})}-1}{\lambda_{l,f}-\lambda_{k,f_0}}=D\frac{e^{-D\lambda_{k,f_0}}-e^{-D\lambda_{l,f}}}{D(\lambda_{l,f}-\lambda_{k,f_0})}=De^{\xi(\lambda_{k,f_0},\lambda_{l,f})}$$
for some mean values $\xi(\lambda_{k,f_0},\lambda_{l,f})$ between $-D\lambda_{k,f_0}$ and $-D\lambda_{l,f}$ arising from the mean value theorem applied to the exponential map. This remains true in the degenerate case where $\lambda_{k,f_0}=\lambda_{l,f}$ as then $b_{k,l}=De^{-D\lambda_{k,f_0}}$.

By the Weyl asymptotics (\ref{eq:weyl}), we see that for $k\le\kappa$, $l\le K$ and any $K$ fixed, the last displayed exponential is bounded below by a fixed constant depending on $f_{\rm min}, U, K$ and $\kappa$. The second to last term in the previous display is of order $\lambda_{l,f}^{-1}$ for fixed $D$, $k \le \kappa$, as $l \to \infty$. Hence we have for all $l$ and some $C=C(D,\mathcal{O},f_{\rm min},U,\kappa)$,
$$ |b_{l,k}|\ge C\lambda_{l,f}^{-1}.$$
Combining this estimate with Equation~\eqref{eq:psuedolin} and Parseval's identity gives a preliminary stability estimate:

\begin{lem}\label{lem:transitionlowerb} Suppose $\|f\|_{H^s}+\|f_0\|_{H^s}\le U$ for some $U>0$, $s>\max(d,2+d/2)$ and that $f,f_0\ge f_{\rm min}>0$ in $\mathcal O$. We then have for any $\kappa\in\mathbb{N}$ the inequality
$$\|P_{D,f}-P_{D,f_0}\|_{H^2\rightarrow H^2}\ge \tilde{c} \max_{1\le k\le\kappa} \|\nabla\cdot[(f-f_0)\nabla e_{k,f_0}]\|_{L^2}$$
where $\tilde{c}=\tilde {c}(D,\mathcal{O}, f_{\rm min},s,U,\kappa)$ is a positive constant.
\end{lem}

Note that as $\kappa$ grows, $\tilde{c} \sim e^{-c\kappa^{2/d}}$, so it is important to keep $\kappa$ fixed. A result similar to Lemma \ref{lem:transitionlowerb} is obtained in \cite{N} considering only the first eigenfunction, which combined with the hot-spots conjecture was shown to give a lower bound proportional to $\|f-f_0\|_{L^2}$.  The fact that our bound extends to any finite set of eigenfunctions will be essential in the proofs of Theorems \ref{thm:optimalrate} and \ref{thm:stability} to be given.

\subsection{Proof of Theorem \ref{thm:optimalrate}}

For $d =1$, all open bounded convex domains in $\mathbb R$ are intervals of the form $(a, b)$ and the proof of Theorem 4 in \cite{N} is easily seen to apply. We assume $d \geq 2$ for the rest of the section. We start with the following auxiliary result:

\begin{prop}\label{prop:eigenlemma}
    Let $H$ be a separable Hilbert space. Also, let $T$ be a bounded linear self-adjoint operator on $H$ with discrete spectrum $\sigma(T)$, and let $\kappa$ be an eigenvalue  with eigenspace $V,$ $\dim(V) = a$. If $$\sigma(T)\cap [\kappa - \epsilon, \kappa + \epsilon] = \{\kappa \}$$ for some $\epsilon >0$, then there exists $\delta > 0$ such that for any self-adjoint operator $S$ with $\|T-S\|_{H\to H} < \delta$, there is a unique subspace $W$ such that $\dim(W) = a$, and $S|_{W}$ is diagonalizable with $a$ (possibly repeated) eigenvalues $\lambda_{1, S}, \dots, \lambda_{a, S}$ such that $|\lambda_{i, S}-\kappa|<\epsilon$ for all $i$. Moreover, any eigenspace of $S$ with eigenvalue within $\epsilon$ of $\kappa$ is contained in $W$. Finally, the choice of subspace is continuous in the sense that $\|P_W - P_V \|_{H \to H} \to 0$ as $\|T - S\|_{H\to H} \to 0$, where $P_V$ and $P_W$ denote the orthogonal projection onto $V$ and $W$, respectively.
\end{prop}

\begin{proof}
    Proposition \ref{prop:eigenlemma} is a special case of Theorem 3.16 in Sections IV.3.4-5 in \cite{K}, which simplifies substantially in our setting as we are dealing only with self-adjoint operators on a Hilbert space. In particular, they have real eigenvalues. The condition $$\sigma(T) \cap [\kappa - \epsilon, \kappa + \epsilon] = \{\kappa\}$$ allows us to isolate the value $\kappa$ from the rest of $\sigma(T)$ via a small ball with boundary $\Gamma$ around $\kappa$ in the complex plane. Noting that $\|T - S\|_{H \to H} \to 0$ is equivalent to $\hat \delta (S, T) \to 0$ in the notation of \cite{K}, we can apply Theorem 3.16 from \cite{K} to the decomposition $H = V(T) \oplus M(T)$, where $V(T)$ is the eigenspace of $\kappa$, and where $M(T)$ is some complementary subspace. Thus for $S$ sufficiently close to $T$, we get a corresponding decomposition $$H = V(S) \oplus M(S)$$ such that the spectrum of $S|_{V(S)}$ is contained within $\Gamma$, and also $\dim(V(S)) = \dim(V(T)) = a$. Since $S$ is self-adjoint, we get $a$ (possibly repeated) real eigenvalues in $V(S)$.
\end{proof}

\begin{cor}\label{lem:eigensum}
    In the setting of Proposition \ref{prop:eigenlemma}, given $e\in V$ a normalized eigenvector of $T$, there is a choice of $E \in W$ such that $\|E\| = 1$ and $\|E - e\| \to 0$ as $\|S-T\|_{H \to H} \to 0$.
\end{cor}

\begin{proof}
    Given $e \in V$, consider $P_We$. Then, since $P_W \to P_V$ as $\|S-T\|_{H\to H} \to 0$, we see that $P_We \to P_V e = e$ as $\|S - T\|_{H\to H} \to 0$. Since $\|e\| = 1$, this means in particular that for $\|P_W-P_V \|_{H\to H}$ small enough, $\|P_We\| > 1/2$, say. Hence, we see that $E \coloneq P_W e/ \|P_We\|$ has the desired properties.
\end{proof}

Now we return to the setting of \textsection \ref{sec:importantlemma} with $f, f_0$ satisfying the hypotheses introduced there. First we note that the perturbation result given in Lemma~5 from \cite{N} still holds for general domains, not just the ``rounded cylinder.''

\begin{lem}\label{lem:convergenceOfInverse}
    Regarding $\mathcal L ^{-1}_f, \mathcal L^{-1}_1$ as bounded linear operators on $L_0^2(\mathcal O)$, we have $\|\mathcal L ^{-1}_f - \mathcal L ^{-1}_1 \|_{L_0^2 \to L_0^2} \lesssim \|f - 1 \|_\infty$, where the constant depends only on $f_{\rm min}, \| f\|_\infty$, and  $\mathcal O$.
\end{lem}

In our setting, upon applying Corollary \ref{lem:eigensum} to $\mathcal L^{-1}_1$ we obtain $E_f \in L^2_0$ such that $\|E_f - e_{1, 1}\|_{L^2}\to 0$ as $\|f - 1\|_\infty\to 0$, and such that $E_f$ is a linear combination of $a$ eigenfunctions (counted with multiplicity) of $\mathcal L_f$, where $a$ is the degeneracy of the first eigenvalue of $\mathcal L_1$. Using the notation from Proposition \ref{prop:eigenlemma}, the eigenvalues of these eigenfunctions are within $\epsilon$ of $\lambda_{1, 1}^{-1}$. Note that $a$ and $\epsilon$ depend only on $\mathcal O$. Moreover, we observe that the largest eigenvalue of $\mathcal L_f^{-1}$ is $\lambda^{-1}_{1, f}$, and
\begin{align}
        |\lambda_{1, f}^{-1}-\lambda^{-1}_{1,1}|&= \left|\| \mathcal L_f^{-1}\|_{L^2_0 \to L^2_0} - \| \mathcal L_1^{-1}\|_{L^2_0 \to L^2_0}\right| \nonumber\\
        &\leq \|\mathcal L ^{-1}_f - \mathcal L ^{-1}_1 \|< \epsilon, \label{eq:choiceofE}
\end{align}
for $\|f - 1\|_\infty$ sufficiently small. Since the \textit{first} $a$ eigenvalues of $\mathcal L_f$ are the \textit{largest} $a$ eigenvalues of $\mathcal L_f^{-1}$,  we deduce that $E_f$ is a linear combination of the \textit{first} $a$ eigenfunctions of $\mathcal L_f$, say
\begin{equation}\label{eq:eigensumeq}
        E_f = \sum_{i = 1}^a \alpha_{i, f}e_{i, f}, \quad \sum_{i=1}^a |\alpha_{i, f}|^2 = 1.
\end{equation}
Note that in particular we have $|\alpha_{i, f}| \leq 1$ for every $i$.

\begin{lem}\label{lem:approximation_Ee}
  We have  $ \|E_f - e_{1, 1} \|_{C^2(\mathcal O)} \to 0$ as $\|f - 1\|_\infty \to 0$.
\end{lem}

\begin{proof}

    We know from Corollary \ref{lem:eigensum} and Lemma \ref{lem:convergenceOfInverse} that $\| E_f - e_{1, 1} \|_{L^2} \to 0$ as $\|f - 1 \|_\infty \to 0$. By standard interpolation inequalities (see (91) in \cite{N}), we see that for any $s+1 > \alpha>2 + d/2$ there is some $0 < c(s, \alpha) < 1$ s.t.
    \begin{equation} \label{intpol}
        \|E_f - e_{1, 1}\|_{H^\alpha} \leq \|E_f - e_{1, 1} \|_{L^2}^{c(s, \alpha)}\|E_{f} - e_{1, 1}\|_{H^{s+1}}^{1-c(s, \alpha)}.
    \end{equation}
   By Corollary 1 from \cite{N} we obtain a constant $C = C(\mathcal O, d, s, U, f_{\rm min})$ such that $\|e_{j, f}\|_{H^{s+1}} \leq Cj^{(s+1)/d}$. 
    So, by Equation (\ref{eq:eigensumeq}), we observe
    \begin{equation*}
        \|E_f\|_{H^{s+1}} \leq \sum_{j = 1}^a |\alpha_{j, f}|\|e_{j, f}\|_{H^{s+1}} \leq Ca\cdot a^{(s+1)/d}.
    \end{equation*}
    Since $a$ depends only on $\mathcal O$, this means that the second factor in the r.h.s.\ of (\ref{intpol}) is uniformly bounded, and the conclusion of the lemma follows from the Sobolev embedding $H^\alpha \subset C^2$.
\end{proof}
\noindent\emph{Proof of Theorem \ref{thm:optimalrate}.} We follow the argument as in Theorem 10 of \cite{NGW}, but instead of considering solutions for the standard elliptic Dirichlet problem, we consider eigen-problems for the Neumann Laplacian: Let $e_{1, f_0}$ be the first eigenfunction of $\mathcal L_{f_0}$, where $f_0 \equiv 1$ on $\mathcal O$. Since $\lambda_{1, f_0} > 0$, $e_{1, f_0}$ cannot be constant. Thus, $\nabla e_{1, f_0}$ is not zero on all of $\mathcal O$. Without loss of generality, we assume the first component of $\nabla e_{1, f_0}$ is positive at some $x_0$. Since $\nabla e_{1, f_0}$ is continuous, pick some $C, \varepsilon > 0$ and $\delta > 0$ such that $\partial_1 e_{1, f_0} > \varepsilon$ and $|\partial_i e_{1, f_0}| \leq C$ ($i \neq 1$) on a ball $B_\delta (x_0)$ in $\mathcal O$ of radius $\delta$ centred at $x_0$. Next, choose diffusivities $(f_m:m = 1, \dots M)$ as 
$$f_m \coloneqq f_0 + \eta2^{-j(s + d/2)}\sum_{r =1}^{n_j}\beta_{r, m}\Psi_{j, r},~~j \in \N,$$
where $n_j$ are integers such that $n_j \simeq 2^{jd}$ and the $\Psi_{j, r}$, $r = 1, \dots n_j,$ are orthonormal tensor Daubechies wavelets supported in $B_\delta(x_0)$, chosen as in (4.17) in \cite{NGW}, so that in particular $f_m \equiv f_{m'} \equiv 1$ on $\mathcal O \setminus \overline{B_\delta(x_0)}$. Specifically, the $\partial_1$ derivative of $\Psi_{j,r}$ is scaled to be a constant factor $c$ larger than the other partial derivatives. Both $\eta>0$ and $1/c$ will be chosen small enough below. The coefficients $\beta_{r, m}, m = 1, \dots, M,$ of $f_m$ will be sufficiently separated elements of the discrete hypercube, $\beta_m \in \{-1, 1 \}^{n_j}$, for a suitable sequence $j=j_N\to \infty$ detailed below. By standard wavelet characterisations of Sobolev norms (Ch.~4.3, \cite{GN}), these $f_m$'s lie in a $H^s$-ball centered at $f_0$ of radius $C\eta \le U$ for some fixed $C=C(s)>0$ and $\eta$ small enough. 

Choose $j$ large enough so that $\|f_m-f_0 \|_\infty$ is small enough as in (\ref{eq:choiceofE}) and so that 
\begin{equation}\label{eq:eEC2}
    \| E_{f_m} - e_{1, f_0} \|_{C^2(\mathcal{O})} < \varepsilon/2,\qquad m=1,\dots,M,
\end{equation}
by Lemma~\ref{lem:approximation_Ee}. We see that in particular $\partial_1 E_{f_m} > \varepsilon/2$ and $|\partial_i E_{f_m}| \leq C + \varepsilon/2 = C'$ on $B_\delta(x_0)$. From the statement of Lemma~\ref{lem:transitionlowerb}, setting $f$ to be $f_m$, $f_0$ to be $f_{m'}$ for any $m \neq m'$, and using \eqref{eq:eigensumeq} we see that for $c_a>0$ the constant $\tilde c$ with $\kappa =a$
\begin{align*}
    \frac{c_a}{a}\|\nabla\cdot[(f_m-f_{m'})\nabla E_{f_{m'}}]\|_{L^2} & \leq \frac{c_a}{a}\sum_{k = 1}^a|\alpha_{k, f_{m'}}|\|\nabla\cdot[(f_m-f_{m'})\nabla e_{k, f_{m'}}]\|_{L^2} \\
    &\leq c_a \max_{1\leq k \leq a}\|\nabla\cdot[(f_m-f_{m'})\nabla e_{k, f_{m'}}]\|_{L^2} \\
    &\leq \|P_{D,f_m}-P_{D,f_{m'}}\|_{H^2\rightarrow H^2}.
\end{align*}
 Next, we lower bound out
\begin{align*}
    \|\nabla \cdot [(f_m-f_{m'})\nabla E_{f_{m'}}]\|_{L^2} &\geq \underbrace{\|\nabla(f_m-f_{m'})\cdot\nabla E_{f_{m'}}\|_{L^2}}_{I} - \underbrace{\|(f_m-f_{m'})\Delta E_{f_{m'}}\|_{L^2}}_{II}.
\end{align*}
Following the arguments in \cite{NGW},
\begin{align}
    I &= \|\sum_{i = 1}^d \partial_{i}(f_m - f_{m'})\partial_{i}E_{f_{m'}}\|_{L^2}\nonumber\\
    &\geq \|\partial_{1}(f_m - f_{m'})\partial_{1}E_{f_{m'}}\|_{L^2}-\|\sum_{i = 2}^d \partial_{i}(f_m - f_{m'})\partial_{i}E_{f_{m'}}\|_{L^2}\nonumber\\
    &\geq \frac{\varepsilon}{2}\|\partial_{1}(f_m - f_{m'})\|_{L^2}-C'\|\sum_{i = 2}^d \partial_{i}(f_m - f_{m'})\|_{L^2}.
\end{align}
Next, we note that by construction of the $f_m$, for $i \neq 1$,
$$\|\partial_i(f_m - f_{m'})\|_{L^2} = \frac{1}{c}\|\partial_1(f_m - f_{m'})\|_{L^2}$$
Hence, we observe
\begin{align}
    I &\geq \frac{\varepsilon}{2}\|\partial_{1}(f_m - f_{m'})\|_{L^2}-\frac{C'(d-1)}{c}\| \partial_{1}(f_m - f_{m'})\|_{L^2} \nonumber\\
    &\geq \frac{\varepsilon}{4}\|\partial_{1}(f_m - f_{m'})\|_{L^2},
\end{align}
for a sufficiently large choice of $c$. Finally, choosing the $\beta_m$ vectors as in the display after (4.22) in \cite{NGW} we have
\begin{equation} \label{sepM}
    I \gtrsim \|\partial_{1}(f_m-f_{m'}) \|_{L^2} \gtrsim 2^{-j(s-1)},~~\text{and }~M \simeq 2^{c_1 2^{jd}}, ~c_1>0.
\end{equation}
Next we show $II$ is of smaller order than $I$.  We have
\begin{align*}
    II^2 &\leq \|f_m - f_{m'}\|_{L^2}^2\|E_{f_{m'}}\|^2_{C ^2} = \eta^2 2^{-2j(s + d/2)}\sum_{r = 1}^{n_j}|\beta_{r,m} - \beta_{r, m'}|^2\|E_{f_{m'}}\|^2_{C ^2} \\
    &\lesssim \sup_{m'} \|E_{f_{m'}}\|^2_{C ^2} 2^{-2js}
\end{align*}
Therefore, we need only to  show that $\sup_{m'} \|E_{f_{m'}}\|^2_{\mathcal C ^2}$ can be bounded by a constant independent of $j$, which follows from \eqref{eq:eEC2}. In summary, we have shown:
\begin{equation}\label{sepbd}
    \|P_{D,f_m}-P_{D,f_{m'}}\|_{H^2\rightarrow H^2}\gtrsim 2^{-j(s-1)},~~~m \neq m'.
\end{equation}

Next, let $KL(f, f')$ denote the Kullback-Leibler divergence (see Appendix \ref{sec:finisher}) between the distributions of $(X_0, X_D, \dots, X_{ND})$ and $(X_0', X_D', \dots, X_{ND}')$ corresponding to the solutions of (\ref{eq:SDE}) with diffusivities $f$ and $f'$, respectively. Note that $p_{D, f}$ from (33) in \cite{N} can be regarded as the joint probability density of $(X_0, X_D)$ given diffusivity $f$. For each $i = 0, \dots N$, we know $X_{iD} \sim \Unif(\mathcal O)$, and $(X_t: t\geq 0)$ is a Markov process, so the joint distribution of $(X_0, \dots, X_{ND})$ has density which splits into the product $p_{D, f}(X_0, X_D)\cdots p_{D, f}(X_{(N-1)D}, X_{ND})$. Hence,
\begin{align*}
    KL(f_m, f_0) &= \sum_{i = 0}^{N-1} \mathbb{E}_{f_0}\log \frac{p_{D, f_0}(X_{iD}, X_{(i+1)D})}{p_{D, f_m}(X_{iD}, X_{(i+1)D})}.
\end{align*}
Using standard arguments with the wavelet characterization of the $H^{-1}(\mathbb R^d)$-norm (see p. 370 in \cite{GN}, noting that $B_{2, 2}^{-1} = H^{-1}$), and Theorem 11 in \cite{N}, we have
\begin{equation*}
    KL(f_m, f_0) \lesssim N\|f_m - f_0 \|_{H^{-1}}^2 \simeq \eta^2 N2^{-2j(s+1)}.
\end{equation*}
Now, we choose $2^j \simeq N^{\frac{1}{2s + 2 + d}}$, so that $N2^{-2j(s +1)} \simeq 2^{jd} \lesssim \log M$, where we recall that $M$ from (\ref{sepM}) is the number of `hypotheses' $f_m$ lying in a ball of $H^s$ for which the distances from (\ref{sepbd}) are lower bounded by $N^{(-s+1)/(2s+2+d)}$. Choosing $\eta$ small enough we can apply Theorem \ref{thm:632} from Appendix \ref{sec:finisher} with a sufficiently small $\alpha$ to complete the proof.

\subsection{Proof of Theorem \ref{thm:stability}}

\subsubsection{Stability for a transport operator}\label{subsub:stability}

Consider the transport operator $h\rightarrow\nabla\cdot(h\nabla u)$ acting on $H^1(\mathcal{O})$ functions, where $u\in C^2(\mathcal{O})$ is fixed.  In view of Lemma \ref{lem:transitionlowerb}, we have in mind $u=e_{k,f_0}$ for some $k$ and $h=f-f_0$, so $h$ changes sign over $\mathcal{O}$. To deal with this, let us first consider the operator $\nabla\cdot(|h|\nabla u)$.


\begin{lem}\label{lem:chainrule}
 For any vector field $v\in C^1(\mathcal{O},\R^d)$ and $h\in H^1(\mathcal{O})$ we have $$|\nabla\cdot (|h|v)| \le |\nabla\cdot(hv)|$$ almost everywhere on $\mathcal{O}$. In particular, for $u\in C^2(\mathcal{O})$, $$\|\nabla\cdot(|h|\nabla u)\|_{L^2} \le \|\nabla\cdot(h\nabla u)\|_{L^2}.$$
 \begin{proof}
For $x\in \mathbb{R}$, define
\begin{equation*}
    \text{sign}(x) = \begin{cases}
        1 & x>0, \\ 0 & x=0, \\ -1 & x<0.
    \end{cases}
\end{equation*}
By Lemma 7.6 in \cite{gilbarg}, we have $|h|\in H^1(\mathcal{O})$ and the chain rule holds: $\nabla |h|=\text{sign}(h)\nabla h$ almost everywhere. Thus, by the product rule (see e.g.\ (7.18) in \cite{gilbarg}) we have
\begin{align*}
\nabla\cdot(|h|v) &= \text{sign}(h)\nabla h \cdot v + |h|\nabla\cdot v  \\
&= \text{sign}(h)(\nabla h\cdot v + h\nabla\cdot v) \\
&= \text{sign}(h)\nabla \cdot (hv) \\
\intertext{Hence, since $|\text{sign}(h)| \le 1$, we have}
 |\nabla\cdot(|h|v)|&\le |\nabla\cdot(hv)|
\end{align*}
almost everywhere.
The last conclusion of the lemma follows by taking $v=\nabla u$ and integrating the square of both sides in the last inequality.
\end{proof}
\end{lem}

We can now prove: 

\begin{lem}\label{lem:L1upperb-rev}
 Suppose $u_1,\dots,u_\kappa\in C^2(\mathcal{O})$ are such that $\|u_k\|_{L^2(\mathcal{O})}=1$ for every $k=1,\dots,\kappa$ and that
$$\sum_{k=1}^\kappa |\nabla u_k(x)|^2 \ge c > 0,\qquad x\in\overline{\mathcal{O}_0}$$ for some  subset $\mathcal O_0$ of $\mathcal O$.
Then for all $h\in H^1(\mathcal{O})$ with boundary trace $h_{|\partial \Omega}=0$ we have
$$\|h\|_{L^1(\mathcal{O}_0)}\le c^{-1} \sum_{k=1}^\kappa\|\nabla\cdot(h\nabla u_k)\|_{L^2(\mathcal{O})}.$$
\end{lem}
\begin{proof}
By the Cauchy-Schwarz inequality, we have the upper bound
$$\left|\int_\mathcal{O} u_k\nabla\cdot(|h|\nabla u_k)\right|\le \|u_k\|_{L^2(\mathcal{O})}\|\nabla \cdot (|h| \nabla u_k)\|_{L^2(\mathcal{O})}=\|\nabla \cdot (|h| \nabla u_k)\|_{L^2(\mathcal{O})}.$$
By the divergence theorem (\cite{Taylor}, p.140f.), we have
\[
\left|\int_\mathcal{O} u_k\nabla\cdot (|h|\nabla u_k)\right|=\left| \int_{\mathcal{O}} |h|\nabla u_k\cdot\nabla u_k\right| \ge \int_{\mathcal{O}_0} |h\|\nabla u_k|^2.
\]
Combining these two inequalities, and summing over $k$, by Lemma \ref{lem:chainrule} we obtain
$$c\|h\|_{L^1(\mathcal{O}_0)} \le \sum_{k=1}^\kappa \|\nabla\cdot(|h|\nabla u_k)\|_{L^2(\mathcal{O})}
\le \sum_{k=1}^\kappa \|\nabla\cdot(h\nabla u_k)\|_{L^2(\mathcal{O})},$$
as required.
\end{proof}

\subsubsection{Critical points of eigenfunctions}

We now combine the results of the preceding sections and adapt an idea from Alberti \cite{alberti} (see also \cite{GSAcpde}) to obtain the global Lipschitz stability estimate in Theorem~\ref{thm:stability}. We often suppress $f_0$ in the notation, in particular setting $e_k=e_{k,f_0}$ and $\lambda_{k}=\lambda_{k,f_0}$ for all $k$.
\begin{lem}\label{lem:gradnonzero-rev}
For every compact $\mathcal O_0$ subset of $\mathcal O$ there exist $\kappa\in\N$ and $c>0$ depending only on $\mathcal{O}$, $\mathcal{O}_0$, $s$, $U$ and $f_{\rm min}$ such that
    \[
\sum_{k=1}^\kappa |\nabla e_k(x)|^2\ge c,\qquad x\in\overline{\mathcal{O}_0}.
    \]
\end{lem}
\begin{proof} 
All the constants left implicit in the proof depend only on $\mathcal{O}$, $\mathcal{O}_0$, $s$, $U$ and $f_{\rm min}$. Let $\mathcal{O}'$ be a smooth domain such that $\mathcal{O}_0\Subset\mathcal{O}'\Subset\mathcal{O}$. Let $\chi\in C^\infty_c(\mathcal{O})$ be a cut-off such that $\chi|_{\mathcal{O}'}\equiv 1$. Define
    \[
    \phi(x)=\chi(x) x_1-\frac{1}{|\mathcal{O}|}\int_{\mathcal{O}} \chi(z) z_1\,dz,
    \]
    so that $\phi\in C^\infty_c(\mathcal{O})\cap L^2_0(\mathcal{O})$. Then $\nabla\phi(x) = (1,0,\dots,0)$ for every $x\in\overline{\mathcal{O}_0}$.
    
   Set $l=\lfloor 2+\frac{d}{2}\rfloor>1+\frac{d}{2}$. For $\kappa \ge 1$, set $\phi_\kappa=\phi -\sum_{k=1}^\kappa \langle\phi ,e_k\rangle_{L^2} e_k$. By the Sobolev embedding theorem and by Proposition~\ref{prop:a} we have
    \[
   \left\| \phi_\kappa \right\|^2_{C^1(\overline{\mathcal{O}_0})} 
   \lesssim    \left\| \phi_\kappa\right\|^2_{H^l({\mathcal{O}})}
   \lesssim    \left\| \phi_\kappa \right\|_{\bar H^l_{f_0}}^2 = \sum_{k>\kappa} \lambda_{k}^l \langle\phi ,e_k\rangle_{L^2}^2.
    \]
    Thus
    \[
\left\| \phi_\kappa \right\|_{C^1(\overline{\mathcal{O}_0})}^2 
    \lesssim \lambda_{\kappa}^{-1} \sum_{k>\kappa} \lambda_{k}^{l+1} \langle\phi ,e_k\rangle_{L^2}^2
    \le \lambda_{\kappa}^{-1} \sum_{k=1}^{+\infty} \lambda_{k}^{l+1} \langle\phi ,e_k\rangle_{L^2}^2
   =\frac{\left\| \phi \right\|^2_{\bar H^{l+1}_{f_0}}}{\lambda_{\kappa}}.
    \]
    Now, using that $\left\| \phi \right\|_{\bar H^{l+1}_{f_0}}\lesssim\left\| \phi \right\|_{ H^{l+1}}\lesssim 1$ (again by Proposition~\ref{prop:a}, since $l+1\le s+1$) and the Weyl asymptotics \eqref{eq:weyl} we obtain
    \(
    \left\| \phi_\kappa \right\|_{C^1(\overline{\mathcal{O}_0})}  \le C \kappa^{-1/d}
    \)
    for some $C>0$ depending only on $\mathcal{O}$, $\mathcal{O}_0$, $s$, $U$ and $f_{\rm min}$. Choose $\kappa = \lceil (2C)^d\rceil$, so that
    \[
    \left\| \phi_\kappa \right\|_{C^1(\overline{\mathcal{O}_0})}  \le 1/2.
    \]
As a consequence, for every
     $x\in \overline{\mathcal{O}_0}$ we have
    \[
    |\nabla(\phi -\phi_\kappa)(x)|\ge |\nabla\phi(x)|-|\nabla \phi_\kappa(x)|\ge 1-1/2=1/2.
    \]
    On the other hand, by the Cauchy Schwarz inequality we have
    \[
    |\nabla(\phi -\phi_\kappa)(x)| =  |\sum_{k=1}^\kappa \langle\phi ,e_k\rangle_{L^2} \nabla e_k(x)|\le \left(\sum_{k=1}^\kappa \langle\phi ,e_k\rangle_{L^2}^2\right)^{\frac12}
    \left(\sum_{k=1}^\kappa |\nabla e_k(x)|^2.  \right)^{\frac12}
    \]
    Observing that $\sum_{k=1}^\kappa \langle\phi ,e_k\rangle_{L^2}^2 \le \|\phi\|^2_{L^2}\lesssim1$ concludes the proof.
\end{proof}

Just as in the proof of Lemma~\ref{lem:approximation_Ee}, the Sobolev embedding combined with $||e_k||_{H^{a}} <\infty$ for all $a \le s+1$ (in view of Corollary 1 in \cite{N}) gives:

\begin{lem}\label{lem:ctseigen} We have $e_k\in C^2(\mathcal{O})$ for all $k$.
\end{lem}

\subsubsection{Completion of the proof of Theorem \ref{thm:stability}}
By Lemma~\ref{lem:gradnonzero-rev}, there exist $\kappa\in\N$ and $c>0$ depending only on $\mathcal{O}$, $\mathcal{O}_0$, $s$, $U$ and $f_{\rm min}$ such that
    \[
\sum_{k=1}^\kappa |\nabla e_k(x)|^2\ge c,\qquad x\in\overline{\mathcal{O}_0}.
    \] 
By Lemma \ref{lem:ctseigen}, we have that $e_k$ is $C^2$ for every $k$. Then, by Lemma~\ref{lem:L1upperb-rev} applied with $u_k=e_k$ and $h=f-f_0$, we obtain
$$\|f-f_0\|_{L^1(\mathcal{O}_0)}\le \frac1c\sum_{k=1}^\kappa\|\nabla\cdot((f-f_0)\nabla e_k)\|_{L^2}\le \frac\kappa c \max_{1\le k\le\kappa} \|\nabla\cdot((f-f_0)\nabla e_k)\|_{L^2}.
$$
Recalling that $f=f_0$ in $\mathcal{O}\setminus\mathcal{O}_0$,  Lemma~\ref{lem:transitionlowerb} yields the result.

\section*{Acknowledgements} 
The authors would like to thank the anonymous referee for their very helpful remarks and suggestions.

GSA was co-funded by the European Union (ERC, SAMPDE, 101041040). Views and opinions expressed are however those of the authors only and do not necessarily reflect those of the European Union or the European Research Council. Neither the European Union nor the granting authority can be held responsible for them.  Co-funded by the EU – Next Generation EU. GSA was further supported by the MIUR Excellence Department Project awarded to Dipartimento di Matematica, Università di Genova, CUP D33C23001110001.
GSA is a  member of the ``Gruppo Nazionale per l’Analisi Matematica, la Probabilità e le loro Applicazioni'', of the ``Istituto Nazionale di Alta Matematica''. 

DB was supported by Trinity College, University of Cambridge. RN and AJ were supported by ERC Advanced Grant (Horizon Europe UKRI G116786) and RN was further supported by EPSRC grant EP/V026259.

\appendix
\section{Proposition 2 from \cite{N}}\label{app:Prop2}
The following proposition gives a relationship between the usual Sobolev spaces $H^k$ and the spectral Sobolev spaces $\bar{H}_f^k$ defined in Section~\ref{sec:proofs}. The space $H^k_c$ from \cite{N} is the subspace of $H^k$ of functions which `vanish near the boundary' in the sense defined in \textsection\ref{subsub:stability}. Let us also for convenience denote by $H^1_\nu(\mathcal{O})$ the space of those functions in $H^1$ whose inward normal derivatives exist in the trace sense (e.g., belong to $H^{3/2}$) and vanish a.e.~at the boundary.
\begin{prop}\label{prop:a}
    Let $\mathcal O$ be a bounded convex domain in $\R^d$ with smooth boundary and let $f \in C^1(\mathcal O)$ be s.t.~$\inf_{x \in \mathcal O}f(x) \ge f_{\rm min}>0$. Then $\bar H^1_f(\mathcal O) = H^1(\mathcal O) \cap L_0^2$ and 
\begin{equation*} \label{pain}
\bar H^2_f = H^2 \cap H^1_\nu \cap L^2_0 = \big\{h \in L^2_0: \mathcal L_f h \in L^2_0, (\partial h/\partial \nu) =0 \text{ on } \partial \mathcal O\big\}.
\end{equation*}
If we assume in addition that for some integer $k \ge 2$ either A)  $\|f\|_{C^{k-1}} \le U$ or B) $\|f\|_{H^s} \le U$ for some $s >d$ s.t.~$k \le s+1$, then we have $$\bar H^k_f(\mathcal O) \subset H^k(\mathcal O) \text{ and } \|\phi\|_{H^k} \simeq \|\phi\|_{\bar H^k_f}~~\text{for }\phi \in \bar H^k_f.$$ We further have the embedding $H^k_c \cap L^2_0 \subset \bar H^k_1$ and also if $H^k_c$ is replaced by $H^k_c / \R$ (modulo constants). Finally we have $\bar H^k_f = \bar H^k_{f'}$ for any pair $f,f'$ satisfying A) or B), with equivalent norms. All embedding/equivalence constants depend only on $f_{\rm min}$, $U$, $d$, $k$, and $\mathcal O$.
\end{prop}

\section{A general mini-max lower bound}\label{sec:finisher}

Let $P$ and $Q$ be probability laws on the same space. Then, $KL(P, Q)$ denotes the Kullback-Leibler divergence from $P$ to $Q$ and is given by 
\[KL(P, Q) = \mathbb{E}_{P}\left[\log\frac{dP}{dQ}\right],\]
where $\frac{dP}{dQ}$ is the Radon-Nikodym derivative of $P$ with respect to $Q$.

\begin{thm}\label{thm:632}
    Suppose that a parameter space $\mathcal F$ has a metric $d$ and associated probability laws $\{P_{f}: f \in \mathcal F\}$. Next, suppose $\mathcal F$ contains
    $$\{f_m : m = 0, 1, \dots, M\}, \quad M \geq 3,$$
    which are $2r$ separated, i.e. $d(f_i, f_j) \geq 2r$ for $i \neq j$, and such that the $P_{f_m}$ are all absolutely continuous with respect to $P_{f_0}$. Assume that, for some $\alpha > 0$, 
    $$\frac{1}{M}\sum_{m = 1}^M KL(P_{f_m}, P_{f_0}) \leq \alpha \log M,$$
    Then, 
    $$\inf_{\Tilde f} \sup_{f\in \mathcal F} P_{f}(d(\Tilde f, f) > r) \geq \frac{\sqrt M}{1 + \sqrt M}\left(1-2\alpha -\sqrt{\frac{8\alpha}{\log M}}\right),$$
where the infimum ranges over all estimators (measurable maps into $\mathcal F$) $\Tilde f$.
\end{thm}

A proof can be found in Theorem 6.3.2 from \cite{GN}, see also \cite{T09}.

\end{document}